\documentclass[english]{amsart}

\usepackage{esint}
\usepackage[svgnames]{xcolor} 
\usepackage[colorlinks,citecolor=red,pagebackref,hypertexnames=false,breaklinks]{hyperref}
\usepackage{pgf,tikz}
\usepackage{pdfsync}

\usepackage{dsfont}
\usepackage{url}
\usepackage[utf8]{inputenc}
\usepackage[T1]{fontenc}
\usepackage{lmodern}
\usepackage{babel}
\usepackage{mathtools}  
\usepackage{amssymb}
\usepackage{lipsum}
\usepackage{mathrsfs}
\usepackage{color}

\newtheorem{theorem}{Theorem}[section]
\newtheorem{proposition}{Proposition}[section]
\newtheorem{lemma}{Lemma}[section]

\newtheorem{corollary}{Corollary}[section]

\newtheorem{remark}{Remark}[section]

\numberwithin{equation}{section}

\title[Uniqueness of continuation for semilinear elliptic equations]{Uniqueness of continuation for semilinear elliptic equations}

\author[Mourad Choulli]{Mourad Choulli}
\address{Universit\'e de Lorraine}
\email{mourad.choulli@univ-lorraine.fr}

\thanks{The author is supported by the grant ANR-17-CE40-0029 of the French National Research Agency ANR (project MultiOnde). }

\date{}

\begin{document}

\frenchspacing

\begin{abstract}
We quantify the uniqueness of continuation from Cauchy or interior data. Our approach consists in extending the existing results in the linear case. As by product we obtain a new stability estimate in the linear case. We also show the so-called strong uniqueness of continuation and the uniqueness of continuation from a set of positive measure. These results are derived by using a linearization procedure.
\end{abstract}



\maketitle

\tableofcontents

\section{Introduction}

\subsection{Quantitative uniqueness of continuation from Cauchy or interior data}\label{sb1.1}

Throughout this subsection, $\Omega$ is bounded $C^{1,1}$ domain of $\mathbb{R}^n$, $n\ge 2$, with boundary $\Gamma$. The unit normal exterior vector field on $\Gamma$ is denoted by $\nu$.

Let $A=(a^{k\ell})$  denotes a symmetric matrix with coefficients $a^{k\ell}\in C^{2,1}(\overline{\Omega})$, $1\le k,\ell\le n$, and there exists a constant $\varkappa> 1$ so that
\[
\varkappa ^{-1}|\xi| ^2\le A(x)\xi\cdot \xi \quad x\in \overline{\Omega} ,\; \xi\in \mathbb{R}^n,\qquad \max_{1\le k,\ell\le n}\|a^{k\ell}\|_{C^{2,1}(\overline{\Omega})}\le \varkappa.
\]

Fix a measurable function $f:\Omega \times \mathbb{R}^{n+1}\rightarrow \mathbb{R}$  satisfying: for each $\delta>0$, there exists an nonnegative function  $g_\delta\in L^\infty (\Omega)$ such that
\[
|f(\cdot,z)|\le g_\delta|z|,\quad |z|\le \delta.
\]

This condition is satisfied for instance if $f(\cdot,0)=0$ and $f$ is Lipschitz continuous on bounded subsets of $\mathbb{R}^{n+1}$ with respect to $z$, uniformly in $x\in \Omega$.

Denote by $\mathcal{H}$ the closure of $C^\infty(\overline{\Omega})$ with respect to the norm
\[
\|u\|_{\mathcal{H}}=\|u\|_{H^1(\Omega)}+\|\mathrm{div}(A\nabla u)\|_{L^2(\Omega)}+\|u\|_{L^2(\Gamma)}+\|\nabla u\|_{L^2(\Gamma)}.
\]
Set $\mathbf{B}_\delta=\{u\in W^{1,\infty}(\Omega);\; \|u\|_{W^{1,\infty}(\Omega)}\le \delta\}$ and 
\[
\mathcal{K}_\delta=\mathcal{H}\cap\mathbf{B}_\delta,\quad \delta>0. 
\]

Define the semilinear elliptic operator $\mathscr{E}$ as follows
\[
\mathscr{E}u=\mathrm{div}(A\nabla u)+f(\cdot,u(\cdot),\nabla u(\cdot)),\quad  u\in C^2(\Omega).
\]

Let 
\[
\mathcal{H}_\eta =
\left\{
\begin{array}{ll}
\left\{u\in H^{3/2+\eta}(\Omega);\; \mathrm{div}(A\nabla u)\in L^2(\Omega)\right\},\quad &0<\eta <1/2,
\\
\mathcal{H}_\eta=H^{3/2+\eta}(\Omega), &\eta \ge 1/2.
\end{array}
\right.
\]
Then
\[
\mathcal{H}\supset \bigcup_{\eta >0}\mathcal{H}_\eta.
\]
Define 
\[
\mathcal{K}_{\eta,\delta}=\mathcal{H}_\eta\cap \mathbf{B}_\delta,\quad \eta >0 ,\; \delta>0.
\]

Fix $\Sigma$  a nonempty open subset of $\Gamma$, $\omega\Subset \Omega$ and set
\begin{align*}
&\mathcal{C}(u)=\|u\|_{L^2(\Sigma)}+\|\nabla u\|_{L^2(\Sigma)}+\|\mathscr{E}u\|_{L^2(\Omega)},\quad u\in \mathcal{H},
\\
&\mathcal{I}(u)=\|u\|_{L^2(\omega)}+\|\mathscr{E}u\|_{L^2(\Omega)},\quad u\in \mathcal{H}.
\end{align*}

\begin{theorem}\label{theorem.s1}
Let $0<s<1/2$, $\delta >0$ and $0<\eta_0<\eta_1\le 1/2$. For any $0<r <1$ and $u\in \mathcal{K}_{\eta_1,\delta}$, we have
\begin{equation}\label{s1.4}
C\|u\|_{H^{3/2+\eta_0}(\Omega)}\le r^{\zeta}\|u\|_{H^{3/2+\eta_1}(\Omega)}+e^{\mathfrak{b}r^{-\gamma}}\mathcal{C}(u),
\end{equation}
where $C=C(n,\Omega,\varkappa, \Sigma,\delta, s,\eta_0,\eta_1)>0$, $\mathfrak{b}=\mathfrak{b}(n,\Omega,\varkappa,\Sigma,\delta, s,\eta_0,\eta_1)>0$, $\gamma=\gamma(n,\Omega,\varkappa,\delta)>0$, and $\zeta =\min ((\eta_1-\eta_0)/(1+\eta_0),s/4)$.
\end{theorem}

If $\mathfrak{b}$, $\zeta$ and $\gamma$ are as in the preceding theorem and $\alpha=\zeta/\gamma$, consider the function $\Phi $ defined as follows
\[
\Phi (\rho)= 
\left\{
\begin{array}{ll}
(\ln \rho)^{-\alpha}\quad &\rho \ge e^{\mathfrak{b}},
\\
\rho , &0<\rho <e^{\mathfrak{b}},
\\
0, &\rho=0.
\end{array}
\right.
\]

Minimizing with respect to $r$  the right hand side of \eqref{s1.4}, we obtain the following consequence of Theorem \ref{theorem.s1}.

\begin{corollary}\label{corollary.s1}
Under the assumptions and the notations of Theorem \ref{theorem.s1}, we have 
\[
C\|u\|_{H^{3/2+\eta_0}(\Omega)}\le \Phi\left(\frac{\|u\|_{H^{3/2+\eta_1}(\Omega)}}{\mathcal{C}(u)}\right)
\]
for each $u\in \mathcal{K}_{\eta_1,\delta}$, where $C=C(n,\Omega,\varkappa, \Sigma,\delta,s,\eta_0,\eta_1)>0$.
\end{corollary}

Another consequence of Theorem \ref{theorem.s1} is the following uniqueness of continuation from the Cauchy data on $\Sigma$.

\begin{corollary}\label{corollary.s2}
Let $\eta >0$. If $u\in H^{3/2+\eta}(\Omega)\cap W^{1,\infty}(\Omega)$ satisfies $\mathcal{C}(u)=0$ then $u=0$.
\end{corollary}

We have similar results concerning the uniqueness of continuation from the interior data in $\omega$.

\begin{theorem}\label{theorem.s2}
Let $0<s<1/2$, $\delta >0$ and $0<\eta_0<\eta_1\le 1/2$. For any $0<r <1$ and $u\in \mathcal{K}_{\eta_1,\delta}$, we have
\[
C\|u\|_{H^{3/2+\eta_0}(\Omega)}\le r^{\zeta}\|u\|_{H^{3/2+\eta_1}(\Omega)}+e^{\mathfrak{b}r^{-\gamma}}\mathcal{I}(u),
\]
where $C=C(n,\Omega,\varkappa, \omega,s,\delta,\eta_0,\eta_1)>0$, $\mathfrak{b}=\mathfrak{b}(n,\Omega,\varkappa,\omega,\delta, s,\eta_0,\eta_1)>0$, $\gamma=\gamma(n,\Omega,\varkappa,\delta)>0$, and $\zeta =\min ((\eta_1-\eta_0)/(1+\eta_0),s/4)$.
\end{theorem}

As for Theorem \ref{theorem.s1}, we have the following consequences of Theorem \ref{theorem.s2}.
\begin{corollary}\label{corollary.s3}
Under the assumptions and the notations of Theorem \ref{theorem.s2}, we have 
\[
C\|u\|_{H^{3/2+\eta_0}(\Omega)}\le \Phi\left(\frac{\|u\|_{H^{3/2+\eta_1}(\Omega)}}{\mathcal{I}(u)}\right)
\]
for each $u\in \mathcal{K}_{\eta_1,\delta}$, where $C=C(n,\Omega,\varkappa, \omega,\delta,s,\eta_0,\eta_1)>0$.
\end{corollary}

\begin{corollary}\label{corollary.s4}
Let $\eta >0$. If $u\in H^{3/2+\eta}(\Omega)\cap W^{1,\infty}(\Omega)$ satisfies $\mathcal{I}(u)=0$ then $u=0$.
\end{corollary}

\subsection{Strong uniqueness of continuation}\label{sb1.2}

In this subsection $\Omega$ is bounded domain of $\mathbb{R}^n$, $n\ge 3$, and $A=(a^{k\ell})$  denotes a symmetric matrix with coefficients $a^{k\ell}\in C^{0,1}(\overline{\Omega})$, $1\le k,\ell\le n$, and there exists a constant $\varkappa> 1$ so that
\[
\varkappa ^{-1}|\xi| ^2\le A(x)\xi\cdot \xi \quad x\in \overline{\Omega} ,\; \xi\in \mathbb{R}^n,\qquad \max_{1\le k,\ell\le n}\|a^{k\ell}\|_{C^{0,1}(\overline{\Omega})}\le \varkappa.
\]

Pick a measurable function $f:\Omega \times \mathbb{R}^{n+1}\rightarrow \mathbb{R}$ satisfying $f(x,\cdot)\in C^1(\mathbb{R}^{n+1})$ for a. e. $x\in \Omega$, $f(\cdot ,0)=0$ and  the following conditions hold: for each $\delta >0$, there exist $g_\delta^0\in L^{n/2}(\Omega)$  and $g_\delta^j\in L^s(\Omega)$, $j=1,\ldots, n$, for some fixed $s>n$, such that
\begin{align*}
&|\partial_{z_0}f(\cdot ,z)|\le g_\delta^0,\quad |z|\le \delta,
\\
&|\partial_{z_j}f(\cdot ,z)|\le g_\delta^j,\quad |z|\le \delta,\; j=1,\ldots, n.
\end{align*}

Let $u\in L_{\mathrm{loc}}^2(\Omega)$. We say that $u$ vanishes of infinite order at $x_0\in \Omega$ if for $r>0$ sufficiently small we have
\[
\int_{B(x_0,r)}u^2dx=O(r^N),\quad \mbox{for all}\; N\in \mathbb{N}.
\]

Also, we say that $u\in H^1(\Omega)$ satisfies $\mathscr{E}u=0$ in $\Omega$ if
\[
-\int_\Omega A\nabla u\cdot \nabla vdx+\int_\Omega f(\cdot,u(\cdot),\nabla u(\cdot))vdx=0,\quad  v\in C_0^\infty (\Omega).
\] 

\begin{theorem}\label{theorem.su1}
Let $u\in W^{1,\infty}(\Omega)$ satisfying $\mathscr{E}u=0$ in $\Omega$. If $u$ vanishes of infinite order at $x_0\in \Omega$ then $u=0$.
\end{theorem}

This theorem will serve to prove the uniqueness of continuation from  a set of positive measure.

\begin{theorem}\label{theorem.su2}
Let $E\subset \Omega$ be a set of positive measure and $u\in W^{1,\infty}(\Omega)$ satisfying $\mathscr{E}u=0$ in $\Omega$. If $u=0$ in $E$ then $u=0$.
\end{theorem}

\subsection{Comments}\label{sb1.3}

To the best of our knowledge Theorems \ref{theorem.s1} and \ref{theorem.s2} are the first results quantifying the uniqueness of continuation from the Cauchy or interior data for semilinear elliptic equations. A uniqueness of continuation  across a non characteristic hypersurface was proved in \cite{Do}. Under a sign condition, various uniqueness of continuation results were established in the sublinear case in \cite{Ru, SW}.

Let us discuss the particular case where $f$ is linear with respect to $z$:
\[
f(x,z)=B_0(x)z_0+B'(x)\cdot z',\quad x\in \Omega ,\; z=(z_0,z')\in \mathbb{R}^{n+1}.
\]

If $B_0\in L^\infty(\Omega)$ and $B'\in L^\infty(\Omega;\mathbb{R}^n)$ then one can easily see that Theorem \ref{theorem.s1}, Corollary \ref{corollary.s1}, Theorem \ref{theorem.s2} and Corollary \ref{corollary.s3} remain valid with $u\in \mathcal{K}_{\eta_1,\delta}$ substituted by $u\in \mathcal{H}_{\eta_1}$, and Corollary \ref{corollary.s2} and Corollary \ref{corollary.s4} still hold when $u\in H^{3/2+\eta}(\Omega)\cap W^{1,\infty}(\Omega)$ is replaced by $u\in H^{3/2+\eta}(\Omega)$.

These results generalize those of \cite{Bo} (Laplace operator) and \cite{Ch2020}. We have similar results in the case of a Lipschitz domain and $C^{1,\alpha}$ solutions. We refer to \cite{BD} dealing with the case of the Laplace operator, and \cite{Ch2016,BC} for the operator $\mathscr{E}$.

In fact, these results can be extended to the case where $B_0$ is unbounded. Precisely, if we suppose that $B_0\in L^p(\Omega)$ for some $p>2$ when $n=2$ and $B_0\in L^n (\Omega)$ when $n\ge 3$ then, according to Sobolev's embedding theorem, we get
\begin{equation}\label{s1.5}
\int_\Omega |B_0w|^2dx\le \mathfrak{e}\|B_0\|_{L^{p_n}(\Omega)}^2\|w\|_{H^1(\Omega)}^2,\quad w\in H^1(\Omega),
\end{equation}
where $p_n=p$ for $n=2$ and $p_n=n$ for $n\ge 3$, and $\mathfrak{e}=\mathfrak{e}(n,\Omega)$.

On the other hand, if $\phi$ is the weight function in \eqref{s1.3} then we establish first a Carleman inequality for the conjugate operator $e^{\tau \phi}\mathrm{div}(A\nabla \cdot)e^{-\tau \phi}$. In light of  \eqref{s1.5}, this Carleman inequality still holds when $e^{\tau \phi}\mathrm{div}(A\nabla \cdot)e^{-\tau \phi}$ is substituted by $e^{\tau \phi}\mathscr{E}e^{-\tau \phi}$. This new Carleman inequality yields \eqref{s1.3} in a straightforward manner.

The uniqueness of continuation result in the best possible case $B_0\in L^{n/2}(\Omega)$ and $B'=0$ was obtained in \cite{JK} (see also \cite{So1}). It is worth noting that to handle this case a $L^2$ Carleman estimate is not sufficient. $L^p$ type Carleman inequalities are necessary (see \cite{JK,So1} for more details). An improvement of the result in \cite{So1} can be found in \cite{So2}.

The case $n\ge 3$, $B_0\in L^{n/2}(\Omega)$ and $B'\in L^s (\Omega;\mathbb{R}^n)$, $s >n$ (+ another term) was studied in \cite{MV}. Precisely, the authors quantify in this work, by means of the $L^2$ norm of the solution in an arbitrary interior subdomain,  the uniqueness of continuation from an interior  set of positive measure. 

It is also worth remaking that Theorems \ref{theorem.su1} and \ref{theorem.su2} extend, when $n\ge 3$, the known results in the linear case, again when $B_0\in L^{n/2}(\Omega)$ and $B'\in L^s (\Omega;\mathbb{R}^n)$, $s >n$.

There is a very large literature devoted to uniqueness of continuation for linear elliptic equations. Without being exhaustive, we quote the few references \cite{ARRV, GL1,GL2,JK,KT,Ku,Le,Re,So1,SS,Ve,Wo}.

\section{Proof of the quantification of the uniqueness of continuation}

We first note that the following inequality holds

\begin{align}
\int_\Omega |f(\cdot,u(\cdot),\nabla u(\cdot))|^2wdx\le k_\delta\int_\Omega (|u|^2+|\nabla &u|^2)wdx,\label{s1.1}
\\
&u\in \mathbf{B}_\delta,\; 0\le w\in L^\infty(\Omega),\nonumber
\end{align}
where $k_\delta=\|g_\delta\|_{L^\infty(\Omega)}^2$.

Pick $\psi \in C^4(\Omega)$ non negative and without critical point in $\overline{\Omega}$. Let $\phi=e^{\lambda \psi}$. From \cite[Theorem 4.1]{Ch2021} we have: there exist $c_0=c_0(n,\Omega,\varkappa)$, $\lambda_0=\lambda_0(n,\Omega,\varkappa)$ and $\tau_0=\tau_0(n,\Omega,\varkappa)$ such that, for any $u\in \mathcal{H}$, $\lambda \ge \lambda_0$ and $\tau\ge \tau_0$, we have the following Carleman inequality 

\begin{align}
&c_0\int_\Omega e^{2\tau \phi}\left[\tau^3\lambda ^4\phi^3 |u|^2+\tau \lambda^2  \phi |\nabla u|^2\right]dx \label{s1.2}
\\
&\hskip1cm \le \int_\Omega e^{2\tau \phi}|\mathrm{div}(A\nabla u)|^2dx +\int_\Gamma e^{2\tau \phi}\left[\tau^3\lambda ^3\phi^3|u|^2+\tau \lambda \phi |\nabla u|^2\right]d\sigma .\nonumber
\end{align}

Fix $\delta>0$. Combining \eqref{s1.1} with $w=e^{2\tau \phi}$ and \eqref{s1.2}, we get the following result: there exist $c_1=c_1(n,\Omega,\varkappa,\delta)$, $\lambda_1=\lambda_1(n,\Omega,\varkappa,\delta)$ and $\tau_1=\tau_1(n,\Omega,\varkappa,\delta)$ such that for any $u\in \mathcal{K}_\delta$, $\lambda \ge \lambda_1$ and $\tau\ge \tau_1$, the following Carleman inequality holds

\begin{align}
&c_1\int_\Omega e^{2\tau \phi}\left[\tau^3\lambda ^4\phi^3 |u|^2+\tau \lambda^2  \phi |\nabla u|^2\right]dx \label{s1.3}
\\
&\hskip2cm \le \int_\Omega e^{2\tau \phi}|\mathscr{E}u|^2dx +\int_\Gamma e^{2\tau \phi}\left[\tau^3\lambda ^3\phi^3|u|^2+\tau \lambda \phi |\nabla u|^2\right]d\sigma .\nonumber
\end{align}

Henceforth,  $\delta >0$ is fixed, $C_0=C_0(n,\Omega ,\varkappa,\delta)$, $C_1=C_1(n,\Omega ,\varkappa,\delta,\Sigma)$ and $C_2=C_2(n,\Omega ,\varkappa,\delta,\omega)$ denote generic constants.

Using \eqref{s1.3} we obtain by adapting the proof of \cite[Proposition A.3]{BC} the following result.

\begin{proposition}\label{propositiona1}
There exist a constant $\gamma=\gamma(n,\Omega,\varkappa,\Sigma) >0$ and a ball $B$ in $\mathbb{R}^n$ satisfying $B\cap \Omega \ne \emptyset$, $B\cap (\mathbb{R}^n\setminus\overline{\Omega})\ne\emptyset$ and $B\cap \Gamma\Subset \Sigma$ so that, for any  $u\in \mathcal{K}_\delta$ and $\epsilon >0$, we have
\begin{equation}\label{a4}
C_0\|u\|_{H^1(B\cap \Omega  )}\le \epsilon ^\gamma \|u\|_{H^1(\Omega )} +\epsilon ^{-1}\mathcal{C}(u).
\end{equation}
\end{proposition}

Let $B$ as in the preceding proposition. Pick then $\tilde{x}\in B\cap \Gamma$. As $B\cap \Omega$ is Lipschitz, it contains a cone with vertex at $\tilde{x}$. That is we  find $R>0$, $\theta \in ]0,\pi /2[$ and $\xi \in \mathbb{S}^{n-1}$ so that
\[
\mathcal{C}(\tilde{x})=\left\{x\in \mathbb{R}^n;\; 0<|x-\tilde{x}|<R,\; (x-\tilde{x})\cdot \xi >|x-\tilde{x}|\cos \theta \right\}\subset B\cap \Omega .
\]
Let $x_\varrho = \tilde{x}+ [\varrho/(3\sin \theta)] \xi$, with $\varrho < (3R\sin \theta)/2$. Then $\mbox{dist}(x_\varrho ,\partial (B\cap \Omega))> 3\varrho $.

Define
\begin{align*}
&\Omega^\rho =\{x\in \Omega ;\; \mbox{dist}(x,\Gamma )>\rho\},\quad \rho >0,
\\
&\Omega_\rho =\{x\in \Omega ;\; \mbox{dist}(x,\Gamma )<\rho\},\quad \rho >0,
\end{align*}
and set
\[
\rho ^\ast =\sup \{\rho >0;\; \Omega ^\rho \ne \emptyset\}.
\]

Let $0<\rho \le \rho ^\ast/3$, $u\in \mathcal{K}_\delta$, $y,y_0\in \Omega^{3\rho}$ and $\epsilon >0$. We have similarly to \cite[(4)]{Ch2020} 
\begin{equation}\label{a5}
C_0\|u\|_{L^2(B(y,\rho ))}\le  \epsilon^{\frac{1}{1-\ell (\rho )}}\| u\|_{L^2(\Omega )}+\epsilon^{-\frac{1}{\ell(\rho )}}\left(\|\mathscr{E}u\|_{L^2(\Omega)}+\|u\|_{L^2(B(y_0,\rho ))}\right),
\end{equation}
where $\ell(\rho )=te^{-C_0/\rho}$ with $0<t=t(n,\Omega,\varkappa)<1$.

Let $y\in \Omega^{3\varrho}$, $0<\varrho < \varrho_0:=\min \left(\rho ^\ast/3,(3R\sin \theta)/2\right)$, $\epsilon >0$ and $\epsilon_1 >0$.
Putting together \eqref{a4} and \eqref{a5} with $y_0=x_\varrho$, we find
\begin{align}
C_1\|u\|_{L^2(B(y,\varrho ))}\le  \epsilon^{\frac{1}{1-\ell (\varrho )}}&\| u\|_{L^2(\Omega )}\label{a6}
\\
&+\epsilon^{-\frac{1}{\ell (\varrho )}}\left[\mathscr{E}u\|_{L^2(\Omega)} +\epsilon_1 ^\gamma \|u\|_{H^1(\Omega )}  +\epsilon_1 ^{-1}\mathcal{C}(u)\right].\nonumber
\end{align}
Taking 
\[
\epsilon_1 =\epsilon^{\frac{1}{\gamma \ell (\varrho )(1-\ell(\varrho ))}}
\]
in \eqref{a6}, we get
\begin{equation}\label{a7}
C_1\|u\|_{L^2(B(y,\varrho ))}\le \phi _0(\epsilon ,\varrho)\|u\|_{H^1(\Omega )}
+\phi_1(\epsilon ,\varrho)\mathcal{C}(u),
\end{equation}
where
\begin{align*}
&\phi _0(\epsilon ,\varrho)=\epsilon^{\frac{1}{1-\ell (\varrho )}},
\\
&\phi_1(\epsilon ,\varrho)=\epsilon^{-\frac{1}{\ell (\varrho )}}\max \left(1, \epsilon^{-\frac{1}{\gamma \ell(\varrho )(1-\ell (\varrho ))}}\right).
\end{align*}

As $\Omega^{3\varrho}$ can be covered by at most $k$ balls with center in $\Omega^{3\varrho}$ and radius $\varrho$ with $k=[d/\varrho]^n$, where $d=d(n,\mathrm{diam}(\Omega)>0$, we derive from \eqref{a7}
\begin{equation}\label{a8}
C_1\|u\|_{L^2(\Omega^{3\varrho})}\le \varrho^{-n}\phi _0(\epsilon ,\varrho)\|u\|_{H^1(\Omega )}
+\varrho^{-n}\phi_1(\epsilon ,\varrho)\mathcal{C}(u).
\end{equation}

We shall also need the following proposition borrowed from \cite{Ch2019} (Proposition 2.1).

\begin{proposition}\label{propositiona2}
There exists $\dot{\varrho}=\dot{\varrho}(n,\Omega)$  so that  we have :
\\
$\mathrm{(i)}$ For any $x\in \Omega_{\dot{\varrho}}$, there exists a unique $\mathfrak{p}(x)\in \Gamma$ such that
\[
|x-\mathfrak{p}(x)|=\mathrm{dist}(x,\Gamma),\quad x= \mathfrak{p}(x)-|x-\mathfrak{p}(x)|\nu(\mathfrak{p}(x)).
\]
(ii) If $x\in \Omega_{\dot{\varrho}}$ then $x_t=\mathfrak{p}(x)-t|x-\mathfrak{p}(x)|\nu(\mathfrak{p}(x))\in \Omega_{\dot{\varrho}}$, $t\in ]0,1]$, and
\[
\mathfrak{p}(x_t)=\mathfrak{p}(x),\; |x_t-\mathfrak{p}(x)|=t\mathrm{dist}(x,\Gamma).
\]
\end{proposition}

This proposition combined with the fact that a $C^{1,1}$ domain admits the uniform interior sphere property (e.g. \cite{Ba}) enable us to state the following result.

\begin{lemma}\label{lemmaa1}
If $\dot{\varrho}$ is as in Proposition \ref{propositiona2} then there exists $\overline{\varrho}\le \dot{\varrho}$ and $\theta \in (0,\pi/2)$ so that, for any $x\in \Omega_{\overline{\varrho}}$, we have 
\[
\mathscr{C}(\mathfrak{p}(x))=\left\{y\in \mathbb{R}^n;\; 0<|y-\mathfrak{p}(x)|<\overline{\varrho},\; (y-\mathfrak{p}(x))\cdot \xi >|y-\mathfrak{p}(x)|\cos \theta \right\}\subset \Omega ,
\]
where $\xi=-\nu (\mathfrak{p}(x))$.
\end{lemma}

Taking $\varrho =\overline{\varrho}/9$ in \eqref{a8}, we obtain

\begin{proposition}\label{propositiona3}
Let $\overline{\varrho}$ as in Lemma \ref{lemmaa1}. Then there exist $\gamma_0=\gamma_0(n,\Omega,\varkappa)>0$  and $\gamma_1=\gamma_1(n,\Omega,\varkappa,\Sigma)>0$  so that, for any $\epsilon >0$ and $u\in \mathcal{K}_\delta$, we have
\begin{equation}\label{a9}
C_1\|u\|_{L^2(\Omega^{\overline{\varrho}/3})}\le \epsilon^{\gamma_0}\|u\|_{H^1(\Omega )}
+\epsilon^{-\gamma_1}\mathcal{C}(u).
\end{equation}
\end{proposition}

We are now ready to prove Theorem \ref{theorem.s1}.

\begin{proof}[Proof of Theorem \ref{theorem.s1}]
Let $0<r<\overline{\varrho}/3$ and pick arbitrarily $x\in \Omega^{r}\setminus \Omega^{\overline{\varrho}/3}$. Set $d=\rm{dist}(x,\Gamma)$, $\tilde{x}=\mathfrak{p}(x)$ and
\[
x_0=\tilde{x}-(d+\tilde{d})\nu (\tilde{x}),
\]
Choose $\tilde{d}>0$  in such a way that $2\overline{\varrho}/3<d+\tilde{d}<\overline{\varrho}$. In that case, $x_0\in \Omega_{\overline{\varrho}}$ and $B(x_0,\overline{\varrho}/3)\subset \Omega^{\overline{\varrho}/3}$. Since $x_0-\tilde{x}$ is colinear to $\nu(\tilde{x})$ we derive from Proposition \ref{propositiona2} that $\tilde{x}_0=\tilde{x}$, $\nu(\tilde{x}_0)=\nu(\tilde{x})$ and $d_0=\mathrm{dist}(x_0,\Gamma)=d+\tilde{d}$. 

Let $\mathscr{C}(\tilde{x})$ be the cone given by Lemma \ref{lemmaa1} and $\rho_0=\min [(d_0/3) \sin \theta,\overline{\varrho}/9]$. This choice of $\rho_0$ guarantees  that $B(x_0,3\rho_0)\subset \mathscr{C}(\tilde{x})\cap B(x_0,\overline{\varrho}/3)$. 

Reducing if necessary $\theta$ we may assume that $\sin \theta \le 1/3$. Hence $\rho_0=(d_0/3) \sin \theta$. Define then the sequence of balls $(B(x_k, 3\rho _k))$ as follows
\begin{eqnarray*}
\left\{
\begin{array}{ll}
x_{k+1}=x_k-\alpha _k \xi ,
\\
\rho_{k+1}=\mu \rho_k ,
\\
d_{k+1}=\mu d_k,
\end{array}
\right.
\end{eqnarray*}
where
\[
d_k=|x_k-\tilde{x}|,\quad \rho _k=\varpi d_k,\quad \alpha _k=(1-\mu)d_k ,
\]
with
\[
\varpi =\frac{\sin \theta}{3},\quad \mu =\frac{3-2\sin \theta}{3-\sin \theta}.
\]
This definition guarantees that 
\[
B(x_k,3\rho _k)\subset \mathscr{C}(\tilde{x}) \quad {\rm and}\quad  B(x_{k+1},\rho _{k+1})\subset B(x_k,2\rho _k),\quad k\ge 1.
\]

Let $\epsilon>0$ and $u\in \mathcal{K}_\delta$. We obtain by mimicking the proof of \cite[(A27)]{BC}
\begin{equation}
C_0\|u\|_{L^2(B(x_k,\rho _k))}\le \epsilon^{-\frac{1}{\gamma^k}} \left(\|u\|_{L^2(B(x_0,\rho _0))}+\|\mathscr{E}u\|_{L^2(\Omega)}\right) \label{a10}
+\epsilon^{\frac{1}{1-\gamma ^k}}\|u\|_{L^2(\Omega)},
\end{equation}
where $0<\gamma=\gamma(n,\Omega,\varkappa) <1$.

As the union of the balls $\overline{B}(x_k,\rho _k)$ contains the segment joining $\tilde{x}$ to $x_0$ we find an integer $k_x \ge 1$ so that $x\in \overline{B}(x_{k_x},\rho _{k_x})\setminus B(x_{k_x-1},\rho _{k_x-1})$. We check in a straightforward manner that 
\begin{equation}\label{a10.1}
k_x\le |\ln \mu|^{-1}\ln [\mu^{-1}(1+\tilde{d}/d)]\le h(r)=|\ln \mu|^{-1}\ln [2\mu^{-1}\overline{\varrho}/r].
\end{equation}
Then \eqref{a10} with $k=k_x$ yields
\begin{align}
&C_0\|u\|_{L^2(B(x_{k_x},\rho _{k_x}))}\le \epsilon^{-\frac{1}{\gamma^{k_x}}} \left(\|u\|_{L^2(B(x_0,\overline{\varrho}/3))}+\|\mathscr{E}u\|_{L^2(\Omega)}\right) \label{a13}
\\
&\hskip 7.5cm+\epsilon^{\frac{1}{1-\gamma ^{k_x}}}\|u\|_{L^2(\Omega)}.\nonumber
\end{align}
We get by substituting in \eqref{a13} $\epsilon$ by $\epsilon^{\gamma^{k_x}}$
\begin{align}
&C_0\|u\|_{L^2(B(x_{k_x},\rho _{k_x}))}\le \epsilon^{-1} \left(\|u\|_{L^2(B(x_0,\overline{\varrho}/3))}+\|\mathscr{E}u\|_{L^2(\Omega)}\right) \label{a13.1}
\\
&\hskip 7.5cm+\epsilon^{\frac{\gamma^{k_x}}{1-\gamma ^{k_x}}}\|u\|_{L^2(\Omega)}.\nonumber
\end{align}
Since $B(x_0,\overline{\varrho}/3)\subset \Omega^{\overline{\varrho}/3}$, a combination of \eqref{a9} and \eqref{a13.1} gives
\begin{align}
&C_1\|u\|_{L^2(B(x_{k_x},\rho _{k_x}))}\le \epsilon^{-1}\epsilon_1^{\gamma_0}\|u\|_{H^1(\Omega)}\label{a13.2}
\\
&\hskip 2cm +\epsilon^{\frac{\gamma^{k_x}}{1-\gamma ^{k_x}}}\|u\|_{L^2(\Omega)} +\epsilon^{-1}(\epsilon_1^{-\gamma_1}+1)\mathcal{C}(u),\quad \epsilon_1>0.
\nonumber
\end{align}
Taking in this inequality $\epsilon_1=\epsilon^{1/(\gamma_0(1-\gamma^k))}$, we obtain
\begin{equation}\label{a13.3}
C_1\|u\|_{L^2(B(x_{k_x},\rho _{k_x}))}\le \epsilon^{\frac{\gamma^{k_x}}{1-\gamma ^{k_x}}}\|u\|_{H^1(\Omega)}+\epsilon^{-1}(\epsilon^{-\frac{\gamma_1}{\gamma_0(1-\gamma^{k_x})}}+1)\mathcal{C}(u).
\end{equation}

Let $\tau >0$ to be specified later. Then $\epsilon=r^{\tau \frac{1-\gamma^{k_x}}{\gamma^{k_x}}}$ in \eqref{a13.3} yields
\begin{equation}\label{a13.4}
C_1\|u\|_{L^2(B(x_{k_x},\rho _{k_x}))}\le r^\tau \|u\|_{H^1(\Omega)}+r^{-\tau \frac{1-\gamma^{k_x}}{\gamma^{k_x}}}\left(r^{-\frac{\tau\gamma_1}{\gamma_0\gamma^{k_x}}}+1\right)\mathcal{C}(u).
\end{equation}
In consequence, assuming in addition that $r<1$, we have 
\[
C_1\|u\|_{L^2(B(x_{k_x},\rho _{k_x}))}\le r^\tau \|u\|_{H^1(\Omega)}+r^{-\frac{\mathfrak{a}}{\gamma^{k_x}}}\mathcal{C}(u),
\]
where $\mathfrak{a}=\mathfrak{a}(n,\Omega,\varkappa,\Sigma,\tau)>0$. This and \eqref{a10.1} imply
\begin{equation}\label{a13.5}
C_1\|u\|_{L^2(B(x_{k_x},\rho_{k_x}))}\le r^\tau \|u\|_{H^1(\Omega)}+r^{-\frac{\mathfrak{a}}{h(r)}}\mathcal{C}(u).
\end{equation}
By \eqref{a10.1} we get
\[
\mu^{k_x}\ge \mu(1+\tilde{d}/d)^{-1}\ge  \mu (2\overline{\varrho}/r)^{-1}
\]
and hence
\[
\rho_{k_x}\ge \frac{\sin \theta}{3}(d+\tilde{d})\mu^{k_x}\ge \frac{\mu \sin \theta}{18}r=\rho_r.
\]
As $x$ was chosen arbitrary in $\Omega^r\setminus\Omega^{\overline{\varrho}/3}$ and $\Omega^r\setminus\Omega^{\overline{\varrho}/3}$ can be covered by at most $O(1/r^n)$ balls of the form $B(x_{k_x},\rho_r)$,  \eqref{a13.5} yields
\begin{equation}\label{a13.6}
C_1\|u\|_{L^2(\Omega^r\setminus\Omega^{\overline{\varrho}/3})}\le r^{-n+\tau} \|u\|_{H^1(\Omega)}+r^{-n-\frac{\mathfrak{a}}{h(r)}}\mathcal{C}(u).
\end{equation}
Fix $0<s<1/2$ arbitrarily and put $\tau =n+s$ in order to get
\begin{equation}\label{a13.7}
C_1\|u\|_{L^2(\Omega^r\setminus\Omega^{\overline{\varrho}/3})}\le r^{s} \|u\|_{H^1(\Omega)}+r^{-n-\frac{\mathfrak{a}}{h(r)}}\mathcal{C}(u).
\end{equation}

On the other hand, we have from \cite[(8)]{Ch2020} (a consequence of Hardy's inequality)
\begin{equation}\label{a13.8}
\|u\|_{L^2(\Omega_r)}\le \mathfrak{c}r^s\|u\|_{H^1(\Omega)},
\end{equation}
where $\mathfrak{c}=\mathfrak{c}(n,\Omega ,s)>0$.

From now on, $C'_1=C'_1(n,\Omega ,\varkappa,\delta,\Sigma,s)>0$ is a generic constant.

Combining  \eqref{a13.7} and \eqref{a13.8}, we derive
\[
C'_1\|u\|_{L^2(\Omega\setminus\Omega^{\overline{\varrho}/3})}\le r^s\|u\|_{H^1(\Omega)}+r^{-n-\frac{\mathfrak{a}}{h(r)}}\mathcal{C}(u).
\]
Hence
\[
C'_1\|u\|_{L^2(\Omega\setminus\Omega^{\overline{\varrho}/3})}\le r^s\|u\|_{H^1(\Omega)}+r^{-\mathfrak{b}r^{-\beta}}\mathcal{C}(u),
\]
where $\mathfrak{b}=\mathfrak{b}(n,\Omega,\varkappa,\Sigma ,s)>0$  and $\beta =|\ln \gamma||\ln \mu|^{-1}$.

This and \eqref{a9} with $\epsilon$ chosen in such a way that $\epsilon^{\gamma_0}=r^s$ give, by modifying if necessary $\mathfrak{b}$,
\begin{equation}\label{a13.9}
C'_1\|u\|_{L^2(\Omega)}\le r^s\|u\|_{H^1(\Omega)}+r^{-\mathfrak{b}r^{-\beta}}\mathcal{C}(u).
\end{equation}

Now, in light of the interpolation inequality
\[
\mathbf{c}\|u\|_{H^{1/2}(\Omega)}\le \epsilon \|u\|_{H^1(\Omega)}+\epsilon^{-1}\|u\|_{L^2(\Omega)},\quad \epsilon >0,
\]
where $\mathbf{c}=\mathbf{c}(n,\Omega)$, we get from \eqref{a13.9}
\[
C'_1\|u\|_{H^{1/2}(\Omega)}\le (\epsilon +\epsilon^{-1}r^s)\|u\|_{H^1(\Omega)}+\epsilon^{-1}r^{-\mathfrak{b}r^{-\beta}}\mathcal{C}(u).
\]
Taking $\epsilon=r^{s/2}$ and substituting $\mathfrak{b}$ by a similar constant, we obtain
\begin{equation}\label{a13.9.1}
C'_1\|u\|_{H^{1/2}(\Omega)}\le r^{s/2}\|u\|_{H^1(\Omega)}+r^{-\mathfrak{b}r^{-\beta}}\mathcal{C}(u).
\end{equation}

Pick $0<\eta_0<\eta_1 \le 1/2$ and let $u\in \mathcal{K}_{\eta_1,\delta}$. From the interpolation inequality
\[
\mathbf{c}_1\|u\|_{H^{3/2+\eta_0}(\Omega)}\le \epsilon^{(\eta_1-\eta_0)/(1+\eta_0)}\|u\|_{H^{3/2+\eta_1}(\Omega)}+\epsilon^{-1}\|u\|_{H^{1/2}(\Omega)},\quad \epsilon >0,
\]
where $\mathbf{c}_1=\mathbf{c}_1(n,\Omega,\eta_0,\eta_1)>0$ (e.g. \cite[Theorem 9.6, page 43]{LM}), and \eqref{a13.9.1}, we get
\begin{equation}\label{a13.9.2}
C'_1\|u\|_{H^{3/2+\eta_0}(\Omega)}\le (\epsilon^{-1}r^{s/2}+\epsilon^{\zeta_0})\|u\|_{H^{3/2+\eta_1}(\Omega)}+\epsilon^{-1}r^{-\mathfrak{b}r^{-\beta}}\mathcal{C}(u),
\end{equation}
where $\zeta_0=(\eta_1-\eta_0)/(1+\eta_0)$ and the constant $C'_1$ may also depends on $\eta$ and $\eta_0$.

If $\zeta=\min(\zeta_0,s/4)$ then $\epsilon=r^{s/4}$ in \eqref{a13.9.2} gives, for some $0<r^\ast <1$,
\begin{equation}\label{a13.9.3}
C'_1\|u\|_{H^{3/2+\eta_0}(\Omega)}\le r^ \zeta \|u\|_{H^{3/2+\eta_1}(\Omega)}+r^{-\mathfrak{b}r^{-\beta}}\mathcal{C}(u),\quad 0<r<r^\ast,
\end{equation}
where the constants $C'_1$ and $\mathfrak{b}$ may depend also on $\eta_0$ and $\eta_1$.

We obtain \eqref{s1.4} by substituting in this inequality $r$ by $r/r^\ast$ and modifying $C'_1$ and $\mathfrak{b}$ by similar constants. 
\end{proof}

\begin{proof}[Proof of Theorem \ref{theorem.s2}]
Follows by modifying  slightly  the proof of Theorem \ref{theorem.s1}. The notations and assumptions are those of the preceding proof.

From \cite[Proposition A2]{BC}, we get
\begin{equation}\label{a14}
C_2\|u\|_{L^2(B(x_0,\overline{\varrho}/3))}\le \epsilon_1^\beta\|u\|_{L^2(\Omega)}+\epsilon_1^{-1}\mathcal{I}(u),\quad \epsilon_1>0,
\end{equation}
where $\beta=\beta (n,\Omega, \varkappa, \delta,\omega)>0$.

This inequality together with \eqref{a13.1} yield
\begin{align}
C_2\|u\|_{L^2(B(x_{k_x},\rho _{k_x}))}\le (\epsilon^{-1}\epsilon_1^{\beta}+\epsilon^{\frac{\gamma^{k_x}}{1-\gamma ^{k_x}}})&\|u\|_{L^2(\Omega)}\label{a15}
\\
& +\epsilon^{-1}(\epsilon_1^{-1}+1)\mathcal{I}(u),\quad \epsilon_1>0.
\nonumber
\end{align}
We omit the rest of the proof which is quite similar to that of Theorem \ref{theorem.s1}. One has only to substitute \eqref{a13.2} by \eqref{a15}.
\end{proof}

\section{Proof of the strong uniqueness of continuation}

We prove the results announced in Subsection \ref{sb1.2}. For $u\in W^{1,\infty}(\Omega)$, set
\begin{align*}
V_u&=\int_0^1\partial_{z_0}f(\cdot ,tu(\cdot),t\nabla u(\cdot))dt,
\\
W_u^j&=\int_0^1\partial_{z_j}f(\cdot ,u(\cdot,tu(\cdot),t\nabla u(\cdot))dt,\quad j=1,\ldots n.
\end{align*}

\begin{lemma}\label{lemma.su1}
For any $u\in W^{1,\infty}(\Omega)$, we have $V_u\in L^{n/2}(\Omega)$ and $W_u^j\in L^s(\Omega)$, $j=1,\ldots n$.
\end{lemma}

\begin{proof}
Pick $u\in W^{1,\infty}(\Omega)$ and let $\delta=\|u\|_{W^{1,\infty}(\Omega)}$. In light of Jensen's inequality, we have
\[
|V_u|^{n/2}\le \int_0^1|\partial_{z_0}f(\cdot ,tu(\cdot),t\nabla u(\cdot))|^{n/2}dt\le (g_\delta^0)^{n/2}.
\]
Hence $V_u\in L^{n/2}$ follows. 

We proceed similarly to prove that $W_u^j\in L^s(\Omega)$, $j=1,\ldots n$.
\end{proof}

Define 
\[
\Omega (\rho)=\{x\in \Omega ;\; \mathrm{dist}(x,\Gamma)>4\rho\},\quad \rho >0.
\]
We only consider $\rho\le \rho_0$, where $\rho_0$ is chosen in such a way to guarantee that $\Omega(\rho)$ is connected.

As in the linear case (e.g. \cite{GL1}), the main ingredient in the proof of the strong uniqueness of continuation property is the following doubling inequality. 

\begin{lemma}\label{lemma.su2}
For every $u\in W^{1,\infty}(\Omega)$ satisfying $\mathscr{E}u=0$ in $\Omega$, there exist $\rho_0=\rho_0(\Omega ,\varkappa , f,u)>0$, $0<\kappa=\kappa(\Omega ,\varkappa , f,u)<1/4$ and $C_u=C_u(\Omega ,\varkappa , f,u)>0$ such that, for any $\rho<\rho_0$, $x_0\in \Omega(\rho)$ and $r<\kappa \rho$, we have
\begin{equation}\label{su1}
\int_{B(x_0,2r)}u^2dx\le C_u \int_{B(x_0,r)}u^2dx.
\end{equation}
\end{lemma}

\begin{proof}
Pick $u\in W^{1,\infty}(\Omega)$. Then is not hard to check that
\begin{equation}\label{su0}
\mathscr{E}u=\mathrm{div}(A\nabla u)+Vu+W\cdot \nabla u,
\end{equation}
where $V=V_u$ and $W=(W_u^1,\ldots ,W_u^n)$.

In light of Lemma \ref{lemma.su1}, \eqref{su1} follows by applying \cite[Proposition 2]{MV}.
\end{proof}

\begin{corollary}\label{corollary.su1}
Let $u\in W^{1,\infty}(\Omega)$ satisfying $\mathscr{E}u=0$ in $\Omega$. If $u$ vanishes of infinite order at $x_0\in \Omega$ then $u$ vanishes in a neighborhood of $x_0$.
\end{corollary}

\begin{proof}
Let $u\in W^{1,\infty}(\Omega)$ satisfying $\mathscr{E}u=0$ in $\Omega$. We use the notations of Lemma \ref{lemma.su2} corresponding to this $u$. Then fix $\rho <\rho_0$ and $r<\kappa \rho$ arbitrarily. For any integer $k\ge 1$, we obtain by iterating the doubling inequality \eqref{su1}
\[
\int_{B(x_0,r)}u^2dx\le C^k \int_{B(x_0,2^{-k}r)}u^2dx,
\]
where $C=C_u$. 

Let $\mu_k>0$ defined by the relation $C^k2^{-n\mu_k}=1$. Then the  inequality above yields

\[
\int_{B(x_0,r)}u^2dx\le 2^{n\mu_k}\int_{B(x_0,2^{-k}r)}u^2dx.
\]
Since $u$ vanishes of infinite order at $x_0$, the right hand side of the inequality above tends to $0$ when $k$ goes to $\infty$. Therefore, $u$ vanishes in $B(x_0,r)$.
\end{proof}

With help of three sphere inequality \cite[Theorem 3]{MV} we can proceed as in the proof of \cite[Proposition 2.28, page 28]{Ch2016} to derive Theorem \ref{theorem.su1} from Corollary \ref{corollary.su1}.

In the rest of this section we prove Theorem \ref{theorem.su2}. Our proof is inspired by that of \cite[Theorem 1.2]{Re}.

Let $E\subset \Omega$ be a set of positive measure. Since almost all points of $E$ are density points, there is at least one density point $x_0\in E$. Therefore, for any $0<\epsilon <1$, we find $r_\epsilon >0$ so that 
\begin{equation}\label{su2}
|E\cap B(x_0,r)||B(x_0,r)|^{-1}\ge 1-\epsilon,\quad |E^c\cap B(x_0,r)||B(x_0,r)|^{-1}\le \epsilon,
\end{equation}
for every $0<r<r_\epsilon$, where $E^c=\Omega \setminus E$.

Pick $u\in W^{1,\infty}(\Omega)$ satisfying $\mathscr{E}u=0$ and $u=0$ in $E$. We find by applying \cite[Lemma 3.4, page 54]{LU} (with $u^2$ instead of $u$)
\[
\int_{B(x_0,r)}u^2dx\le cr^n|E\cap B(x_0,r)|^{-1}|E^c\cap B(x_0,r)|^{1/n}\int_{B(x_0,r)}|\nabla u^2|dx,
\]
where $c=c(n)$.
As $|\nabla u^2|=2|u||\nabla u|$, Cauchy-Schwarz's inequality yields 
\begin{equation}\label{su3}
\int_{B(x_0,r)}u^2dx\le c^2r^{2n}|E\cap B(x_0,r)|^{-2}|E^c\cap B(x_0,r)|^{2/n}\int_{B(x_0,r)}|\nabla u|^2dx.
\end{equation}
We combine \eqref{su3} and Caccioppoli's inequality \cite[(13)]{MV} in order to obtain
\begin{equation}\label{su4}
\int_{B(x_0,r)}u^2dx\le C[r^{2n}|E\cap B(x_0,r)|^{-2}][r^{-n}|E^c\cap B(x_0,r)|]^{2/n}\int_{B(x_0,2r)}u^2dx,
\end{equation}
where $C=C(n,\Omega,\varkappa,f)$.

Putting together \eqref{su2} and \eqref{su4}, we get
\begin{equation}\label{su5}
\int_{B(x_0,r)}u^2dx\le C(1-\epsilon)^{-2}\epsilon^{2/n}\int_{B(x_0,2r)}u^2dx,\quad 0<r\le r_\epsilon.
\end{equation}

Let $N\in \mathbb{N}\setminus\{0\}$ and choose $\epsilon$ in such a way that $C(1-\epsilon)^{-2}\epsilon^{2/n}=2^{-N}$. In that case it is more convenient to denote $r_\epsilon$ by $r_N$.

Inequality \eqref{su5} then implies
\begin{equation}\label{su6}
\int_{B(x_0,r)}u^2dx\le 2^{-N}\int_{B(x_0,2r)}u^2dx,\quad 0<r\le r_N.
\end{equation}

If $\mathcal{I}(r)= \int_{B(x_0,r)}u^2dx$ then \eqref{su6} can be rewritten in the following form
\[
\mathcal{I}(r)\le 2^{-N}\mathcal{I}(2r),\quad 0<r\le r_N.
\]
Iterating this inequality, we obtain
\begin{equation}\label{su7}
\mathcal{I}(r)\le 2^{-kN}\mathcal{I}(2^kr),\quad k\in \mathbb{N}\setminus\{0\},\; 0<2^{k-1}r\le r_N.
\end{equation}

Fix $r\le r_N$ and choose $k$ in such a way that $2^{-k}r_N\le r\le 2^{-k+1}r_N$. Noting that $t\rightarrow \mathcal{I}(t)$ is a non decreasing function, we derive from \eqref{su7}
\[
\mathcal{I}(r)\le 2^{-kN}\mathcal{I}(2r_N)
\]
and hence
\[
\mathcal{I}(r)\le (r/r_N)^{N}\mathcal{I}(2r_N).
\]
This means that $\mathcal{I}(r)=O(r^N)$. In other words, $u$ vanishes of infinite order at $x_0$. Theorem \ref{theorem.su2} follows then from Theorem \ref{theorem.su1}.

\begin{remark}\label{remark.su1}
{\rm
\textbf{(i)} In light of identity \eqref{su0}, we derive readily from \cite[Theorem 1]{MV}
\begin{theorem}\label{theorem.su3}
Assume that $\Omega$ is Lipschitz. Let $\delta >0$ and $E\subset \Omega(\rho)$, $\rho\le \rho_0$, be a set of positive measure. For every $u\in \mathbf{B}_\delta$ satisfying $\|u\|_{L^2(\Omega)}\le 1$ and $\mathscr{E}u=0$ in $\Omega$, we have
\[
\|u\|_{L^2(\Omega(\rho))}\le C|\ln \|u\|_{L^2(E)}|^{-\gamma},
\]
where $C=C(n,\Omega,\varkappa, f,\rho,\delta)$ and $\gamma=\gamma (n,\Omega,\varkappa, f,\rho,\delta)$.
\end{theorem}

\textbf{(ii)} Suppose that $f$ is independent of $z'$ and satisfies
\[
|f(\cdot ,z_0)|\le g|z_0|^{r},\quad z_0\in \mathbb{R}, 
\]
where $r\in (0,1+4/(n-2))$ and $g\in L^\infty(\Omega)$, and define
\[
V_u(x)=\frac{f(x,u(x))}{u(x)},\; x\in \{u\ne 0\},\quad V_u(x)=0,\; x\in \{u=0\},\quad u\in H^1(\Omega).
\]
It is easy to check that $V_u\in L^{n/2}(\Omega)$, for every $u\in H^1(\Omega)$, and $f(\cdot,u(\cdot))=V_uu$.

Modifying slightly the proof of Theorems \ref{theorem.su1} and \ref{theorem.su2}, we derive 
\begin{theorem}\label{theorem.su4}
Let $u\in H^1(\Omega)$ satisfies $\mathscr{E}u=0$ in $\Omega$.
\\
(a) If $u$ vanishes of infinite order at $x_0\in \Omega$ then $u=0$.
\\
(b) If $u$ vanishes in a subset of $\Omega$ of positive measure then $u=0$.
\end{theorem}

\textbf{(iii)} The linearization argument used for semilinear equations can be extended to fully nonlinear equations. 

We identify in the sequel a $n\times n$ matrix $m= (m_{ij})$ by the vector 
\[
(m_{11},\ldots, m_{1,n},m_{21},\ldots,m_{2n},\ldots ,m_{n1},\ldots m_{nn}).
\]
Let $F=F(x,y,z,m):\Omega \times\mathbb{R}\times \mathbb{R}^n\times \mathbb{R}^{n\times n} \rightarrow \mathbb{R}$ measurable satisfying $F(x,\cdot )\in C^1(\mathbb{R}\times \mathbb{R}^n\times \mathbb{R}^{n\times n})$ a.e. $x\in \Omega$ and $F(\cdot,0,0,0)=0$. Then the equation
\begin{equation}\label{su8}
F(x,u,\nabla u(x),\nabla^2u(x))=0,\quad \mathrm{in}\; \Omega
\end{equation}
can be rewritten in the form
\[
A_u\nabla^2u+W_u\cdot \nabla u+V_uu=0 \quad \mathrm{in}\; \Omega,
\]
where $\nabla ^2u=(\partial_{ij}^2u)$ is the hessian matrix of $u$. The coefficients of the preceding equation are given by
\begin{align*}
(A_u)_{ij}(x)&=\int_0^1\partial_{m_{ij}}F(x,tu(x),t\nabla u(x),t\nabla^2u(x))dt,\quad 1\le i,j\le n,
\\
(W_u)_j&=\int_0^1\partial_{z_j}F(x,tu(x),t\nabla u(x),t\nabla^2u(x))dt,\quad 1\le j\le n,
\\
V_u&=\int_0^1\partial_yF(x,tu(x),t\nabla u(x),t\nabla^2u(x))dt.
\end{align*}
Set
\begin{align*}
&U_\delta=\{(y,z,m)\in \mathbb{R}\times \mathbb{R}^n\times\mathbb{R}^{n\times n};\; |y|\le \delta,\; \max[z_j|\le \delta ,\; \max|m_{ij}|\le \delta\},\quad \delta >0,
\\
&\tilde{\mathbf{B}}_\delta=\{u\in W^{2,\infty}(\Omega);\|u\|_{W^{2,\infty}(\Omega)}\le \delta \},\quad \delta >0.
\end{align*}
Then Theorems \ref{theorem.su1} and \ref{theorem.su2} still hold when $\mathscr{E}u=0$, $u\in W^{1,\infty}(\Omega)$, is substituted by \eqref{su8} for $u\in W^{2,\infty}(\Omega)$, provided that following assumptions hold: for each $\delta>0$, there exist $\varkappa_\delta >1$, $g_\delta^0\in L^{n/2}(\Omega)$ and $g_\delta^j\in L^s(\Omega)$, for some $s>n$, such that
\begin{align*}
& \partial_{m_{ij}}F(\cdot,u(\cdot),\nabla u(\cdot),\nabla^2u(\cdot))\in C^{0,1}(\Omega),\quad  u\in W^{2,\infty}(\Omega),
\\
&\|F(\cdot,u(\cdot),\nabla u(\cdot),\nabla^2u(\cdot))\|_{C^{0,1}(\Omega)}\le \varkappa_\delta,\quad u\in \tilde{\mathbf{B}}_\delta ,
\\
&(\partial_{m_{ij}}F(x,y,z,m))\xi\cdot \xi \ge \varkappa_\delta^{-1}|\xi|^2,\quad \xi \in \mathbb{R}^n,\; (x,y,z,m)\in \Omega\times U_\delta,
\\
&|\partial_{z_j}F(x,y,z,m)|\le g_\delta^j(x),\quad (x,y,z,m)\in \Omega\times U_\delta,\; j=1,\ldots ,n,
\\
&|\partial_yF(x,y,z,m)|\le g_\delta^0(x),\quad (x,y,z,m)\in \Omega\times U_\delta.
\end{align*}
}
\end{remark}

\vskip .5cm
\end{document}